\documentclass{article}

\usepackage{amsmath,amssymb,url}
\usepackage{enumitem} 
\usepackage{framed, xcolor, nicefrac}
\usepackage{amsopn, amsthm}
\usepackage{tikz-cd}
\usetikzlibrary{graphs, graphs.standard, shapes.geometric, decorations.pathreplacing}

\numberwithin{equation}{section}

\theoremstyle{plain}
\newtheorem{theorem}{Theorem}[section]
\newtheorem{lemma}[theorem]{Lemma}

\theoremstyle{definition}
\newtheorem{definition}[theorem]{Definition}

\newtheorem{remark}[theorem]{Remark}

\renewenvironment{proof}[1][\unskip]{%
\par
\noindent
\textbf{Proof #1.}
\noindent}
{\hfill$\blacksquare$

\bigskip}

\usepackage{etoolbox}

\newcommand{\defqedsymbol}{\hfill\ensuremath{\square}} 

\AtBeginEnvironment{definition}{\pushQED{\defqedsymbol}}
\AtEndEnvironment{definition}{\popQED}

\let\<=\langle
\let\>=\rangle
\def\land{\wedge}       
 
\def\Land{\bigwedge}    
\def\Lor{\bigvee}

\let\notmodels=\nvDash

\renewcommand{\setminus}{\smallsetminus}

\def\eps{\varepsilon}
\def\setminus{-}

\def\M{\mathsf{M}}

\def\Pa{\mathsf{Pa}}

\def\BB{\bar{\mathbf{B}}}

\def\A{\mathbf{A}}
\def\B{\mathbf{B}}
\def\U{\mathbf{U}}

\def\Q{\mathbf{Q}}
\def\K{\mathsf{K}}

\def\Sg{\mathbf{Sg}}
\def\calF{\mathcal{F}}
\def\calG{\mathcal{G}}

\usepackage[T1]{fontenc}
\usepackage[utf8]{inputenc}

\begin{document}


\title{On covers of quasivarieties of p-algebras}

\author{Zal\'an Gyenis \\ {\small \texttt{zalan.gyenis@uj.edu.pl}} \\ {\small Jagiellonian University}}
\date{\today}

\maketitle

\begin{abstract}
	This paper characterizes the covers of varieties of p-algebras in the
	lattice of quasivarieties of p-algebras. In particular, it is shown
	that every such variety has exactly one cover in the lattice of
	subquasivarieties. This answers a problem of 
	Kowalski and Słomczyńska \cite{KowSlom}.
\end{abstract}





\section{Overview}

This article characterizes covers of subvarieties 
in the lattice of subquasivarieties of p-algebras, settling Question 3 of 
Kowalski and Słomczyńska \cite{KowSlom}. I follow the notation and terminology
of \cite{KowSlom}, but to make the paper self-contained, I recall below the
most important facts from \cite{KowSlom} that will be needed. Fundamentals
of universal algebra (say, quasivarieties, subalgebras, etc.) are not recalled, 
but the readers might consult \cite{Bergman,BurrisSankappanavar}. 

P-algebras are bounded distributive lattices endowed with
a unary operation $*$ satisfying the condition $x\land y=0$ iff $x\leq y^*$. 
It is known that p-algebras form a variety $\Pa$, and by a result of 
Lee \cite{Lee} the lattice
of subvarieties of $\Pa$ is a chain
\[
	\Pa_{-1}\subsetneq \Pa_0\subsetneq \Pa_1 \subsetneq \cdots\subsetneq \Pa
\]
of order type $\omega+1$, where $\Pa_{-1}$ is the trivial variety, and
$\Pa_k = \mathbf{HSP}(\BB_k)$ for $k\in \mathbb{N}$. Here
$\BB_k$ is the $k$-atom Boolean algebra with an extra cover of the top 
element. In particular, $\BB_0$ is the two-element Boolean algebra, and $\Pa_0$
is the variety of Boolean algebras. 

$\Pa$ is locally finite (finitely generated p-algebras are finite, 
see \cite{BermanDwinger}), and for any $m>0$ the variety $\Pa_m$ is 
axiomatized by a single identity 
\[
	\Lor_{i=1}^{m+1}(x_i\land\Land_{j\neq i}x_j^*)^* =1.\tag{$\mathbf{ib}_m$}
\]
(see \cite{19} and Lemma 1.3 in \cite{KowSlom}).\\

The main tool that is used in this paper is Priestley-duality 
for finite p-algebras. Objects dual to finite p-algebras are finite 
posets, and morphisms are order-preserving maps commuting with the
maximal elements of the poset. Following \cite{KowSlom}, these morphisms
are called pp-morphisms. Let $(P, \leq)$ be a finite poset, and for $x\in P$
write $\M(x)$ for the set of maximal elements in $P$ above $x$. 
(By the same token, $\M(P)$ is the set of all maximal elements of $P$). 
Then $h:(P,\leq)\to (Q, \leq)$ is a pp-morphism, if it preserves the ordering, 
and $\M(h(x))=h[\M(x)]$ for every $x\in P$. 

For a finite p-algebra $\A\in\Pa$ write $\delta(\A)$ for its dual poset; 
and for a finite poset $P$ write $\eps(P)$ for its dual p-algebra. Elements
of $\delta(\A)$ are the join-irreducible elements of $\A$ ordered by the
converse of the order inherited from $\A$; while elements of $\eps(P)$
are the upsets of the poset. For finite $\A,\B\in\Pa$ 
we have $\A\subseteq \B$ if and only if there is a surjective pp-morphism
$\delta(\A)\twoheadrightarrow\delta(\B)$. Conversely, for finite posets
$P$ and $Q$, there is a surjective pp-morphism $P\twoheadrightarrow Q$
if and only if $\eps(Q)\subseteq \eps(P)$. (Here and later on $\subseteq$ 
between algebras means subalgebra). Recall (see \cite{19}) what 
the dual posets of $\BB_m$ look like:
\begin{figure}[ht!]
    \begin{center}
        \begin{tikzpicture}[thick,scale=0.95]
			\node (A0) at (-2,0.5) {};
			\draw [fill] (A0) circle [radius=.06];
			\draw (-2,-1) node {$\delta(\BB_0)$};

			\node (A1) at (0,0) {};
			\node (A2) at (0,1) {};
			\draw [fill] (A1) circle [radius=.06];
			\draw [fill] (A2) circle [radius=.06];
			\draw (A1.center) -- (A2.center);			
			\draw (0,-1) node {$\delta(\BB_1)$};

			\node (B1) at (2,0) {};
			\node (B2) at (1.5,1) {};
			\node (B3) at (2.5,1) {};
			\draw [fill] (B1) circle [radius=.06];
			\draw [fill] (B2) circle [radius=.06];
			\draw [fill] (B3) circle [radius=.06];
			\draw (B3.center) -- (B1.center) -- (B2.center);			
			\draw (2,-1) node {$\delta(\BB_2)$};

			\node (C1) at (5,0) {};
			\node (C2) at (5-0.5,1) {};
			\node (C3) at (5-1,1) {};
			\node (C4) at (5+1,1) {};
			\node (C5) at (5+0.5,1) {};
			\draw [fill] (C1) circle [radius=.06];
			\draw [fill] (C2) circle [radius=.06];
			\draw [fill] (C3) circle [radius=.06];
			\draw [fill] (C4) circle [radius=.06];
			\draw [fill] (C5) circle [radius=.06];
			\draw (C1.center) -- (C2.center);			
			\draw (C1.center) -- (C3.center);			
			\draw (C1.center) -- (C4.center);			
			\draw (C1.center) -- (C5.center);			
			\draw (5,-1) node {$\delta(\BB_m)$};

			\node at (5, 1) {$\ldots$};
			
			\draw [decorate, decoration={brace, amplitude=5pt}] 
			    ([yshift=0.2cm]C3.center) -- ([yshift=0.2cm]C4.center) 
			    node [midway, above, yshift=0.15cm] {$m$ points};

		\end{tikzpicture}
		\caption{The dual posets $\delta(\BB_m)$ of the p-algebras
		$\BB_m$ generating $\Pa_m$.}
		\label{fig:BBm}
	\end{center}
\end{figure}

\noindent The following lemma is used frequently. (Below $\uplus$ denotes disjoint union).

\begin{lemma}[Lemma 1.7 in \cite{KowSlom}] \label{1.7}
	Let $\mathcal{P}$ be a family of finite posets, and let $X$
	be a finite poset. Then $\eps(X)\in \Q\big(\{\eps(P): P\in\mathcal{P}\}\big)$
	if and only if for some $P_1, \ldots, P_n\in\mathcal{P}$
	there is a surjective pp-morphism 
	$P_1\uplus\cdots\uplus P_n\twoheadrightarrow X$.
\end{lemma} 

Considering the lattice of quasivarieties of p-algebras, it is known that
the interval $[\Pa_{-1}, \Pa_2]$ is a $4$-element chain. The recent paper 
\cite{KowSlom}
lists $3$ finite posets (see Figure \ref{fig:3poset} below) and claims
that $\Q(\eps(P))$ for any of these three posets is a cover of $\Pa_2$, and \cite{KowSlom}
asks the question:

\medskip
\noindent {\bf Question 3.} \emph{Are there any other covers of $\Pa_2$? More generally, what are the covers of $\Pa_m$?}
\medskip

This paper answers both questions, and corrects the claim in \cite{KowSlom} about the covers of $\Pa_2$.
More precisely, Theorem \ref{thm:unique-cover} states that for every $m\geq 1$, the variety 
$\Pa_m$ has exactly one cover in the lattice of quasivarieties of p-algebras, and this cover is
$\Q\bigl(\BB_m, \U_{m+1} \bigr)$, where $\U_{m+1}$ is the dual algebra of the poset
$U_{m+1} = \bigl(\wp(m+1)\setminus\{\emptyset\},\supseteq\bigr)$ (see Figure \ref{fig:um}).
Combining this with the known result that the interval $[\Pa_{-1}, \Pa_2]$ is a $4$-element chain, gives that $\Pa_m$ has a unique cover for every $m$ (see Figure \ref{fig:bottom}). Section \ref{mketto} discusses the case $m=2$ carefully, and Theorems \ref{thm:uccso-elotti} and \ref{thm:uccso} show that exactly one of the posets from \cite{KowSlom} (see Figure \ref{fig:3poset} below) generate a cover of $\Pa_2$. This corrects a claim in \cite{KowSlom}.

\section{Basic tools}

\begin{definition}
	Let $M$ be a non-empty set and $\calF$ a possibly empty family of 
	non-empty subsets of $M$. Let
	\[
		P(M,\calF) = \bigg\{ M, \{m\}: m\in M \bigg\}\cup \calF,
	\]
	ordered by reverse inclusion: for $A,B\in P(M,\calF)$ put
	$A\leq B$ iff $B\subseteq A$. 
	
	$P(M,\calF)$ thus has the set $M$ as a smallest element, singletons
	$\{m\}$ for $m\in M$ as maximal elements, and possibly some other 
	subsets of $M$ that are listed in $\calF$. We call posets of this form
	\emph{reduced} posets, and $M$ is the base of $P$. 
	If $\calF=\emptyset$, then we might simply write 
	$P(M)$ in place of $P(M,\emptyset)$. 
\end{definition}

\noindent Maximal elements of $P=P(M,\calF)$ are the elements of the set
$\M(P) = \big\{ \{m\}: m\in M \big\}$. Further, for any $a,b\in P$
we have $a\leq b$ iff $\M(a)\supseteq \M(b)$, in particular, 
$a=b$ iff $\M(a)=\M(b)$. 

Note that $\delta(\BB_m)$ is (isomorphic to) the reduced poset $P(M,\emptyset)$
with $|M|=m$.

\begin{lemma} \label{lem:mplus1}
	Let $m\geq 1$. For a finite poset $X$, the following are equivalent:
	\begin{enumerate}
		\item[(1)] $\varepsilon(X)\in \Pa_m$;
		\item[(2)] every $x\in X$ satisfies $|\M(x)|\le m$.
	\end{enumerate}
\end{lemma}
\begin{proof}
	$(1)\Rightarrow(2)$\ 
	Since $\Pa_m=\Q(\BB_m)=\Q(\eps(B_m))$, where $B_m$ is the
	poset with one bottom element and $m$ maximal elements, Lemma \ref{1.7} yields a
	surjective pp-morphism from a finite disjoint union of copies of $B_m$ onto $X$.
	Every point of $B_m$ has at most $m$ maximal elements above it, 
	so by preservation of maxima the same holds in $X$.

	$(2)\Rightarrow(1)$\ 
	For each $x\in X$, choose a surjection from the set of 
	maxima of $B_m$ onto $\M(x)$, and map the bottom of $B_m$ to $x$. 
	This gives a pp-morphism $f_x:B_m\to X$.
	Taking the disjoint union of all these maps yields a surjective pp-morphism
	$\biguplus_{x\in X} B_m \twoheadrightarrow X$.
	By Lemma \ref{1.7}, $\eps(X)\in \Q(\eps(B_m))=\Pa_m$.
\end{proof}

\begin{lemma}\label{lem:egykomponens}
	For a finite poset $X$, the following are equivalent:
	\begin{enumerate}
		\item[(1)] $\Pa_m\subseteq\Q(\eps(X))$;
		\item[(2)] there is a surjective pp-morphism $g: X\twoheadrightarrow \delta(\BB_m)$.
	\end{enumerate}
\end{lemma}
\begin{proof}
	By Lemma \ref{1.7}, $\Pa_m\subseteq\Q(\eps(X))$ if and only if
	there is a surjective pp-morphism 
	$f:\biguplus_i X\twoheadrightarrow \delta(\BB_m)$
	from finitely many copies of $X$ onto $\delta(\BB_m)$. 
	From this $(2)\Rightarrow(1)$ is immediate. For $(1)\Rightarrow(2)$
	Write $\bot$ for the bottom element of $\delta(\BB_m)$, and let 
	$x\in \biguplus_i X$ such that $f(x)=\bot$. 
	Then $f$ restricted to the component to which $x$ belongs is already
	a surjective pp-morphism $g: X\twoheadrightarrow \delta(\BB_m)$.
\end{proof}

\begin{lemma}\label{lemma-1}
	Let $\A$ be a finite p-algebra such that $\A\notin\Pa_m$.
	Then there is a reduced poset $P=P(N,\mathcal F)$, $|N|=m+1$,
	such that $\varepsilon(P)\subseteq \A$.
\end{lemma}
\begin{proof}
	Let $R = \delta(\A)$ be the dual poset of $\A$. 
	Since $\A\notin\Pa_m$, Lemma \ref{lem:mplus1} 
	gives a point $x\in R$ such that $|\M(x)|\geq m+1$.
	Choose a set $N$ with $|N|=m+1$, and choose a surjection
	\[
		\pi:\M(R)\twoheadrightarrow N
	\]
	such that $\pi[\M(x)]=N$. Define
	\[
		P= \big\{\pi[\M(y)]: y\in R\big\},
	\]
	ordered by reverse inclusion.
	Since $\pi$ is surjective, $P$ contains every singleton $\{n\}$,
	$n\in N$. Further, as $\pi[\M(x)]=N$, $P$ 
	also contains $N$. Thus $P$ is a reduced poset with base $N$.
	Define
	\[
		q:R\to P,\qquad q(y)=\pi[\M(y)].
	\]
	The map $q$ is surjective by definition. If $y\leq z$, then
	$\M(z)\subseteq \M(y)$, and therefore
	$q(z)\subseteq q(y)$, so $q$ is order-preserving (the ordering of $P$ 
	is reverse inclusion).
	Moreover, 
	\[
		\M(q(y)) = \bigl\{\{n\}:n\in\pi[\M(y)]\bigr\},
	\]
	while
	\[
		q[\M(y)]	= \bigl\{\{\pi(u)\}:u\in \M(y)\bigr\}.
	\]
	Therefore $\M(q(y)) = q[\M(y)]$, and hence $q$ is a surjective pp-morphism.

	By duality, $\varepsilon(P)\subseteq\varepsilon(R)=\A$.	
\end{proof}


\begin{lemma}\label{lem:relative-comparison}
	Let $m\geq 1$, let $M$ and $N$ be sets of cardinality $m+1$, and 
	let $P(M,\mathcal F)$ and $P(N,\mathcal G)$ be reduced posets.
	The following conditions are equivalent:
	\begin{enumerate}
		\item[(1)] $\Q\bigl(\BB_m,\varepsilon(P(N,\mathcal G))\bigr)$
		$\subseteq$ $\Q\bigl(\BB_m,\varepsilon(P(M,\mathcal F))\bigr)$ 
		\item[(2)] There is a bijection $\sigma:M\to N$ such that
		\[
			\sigma[\mathcal F] \subseteq \mathcal G \cup\{N\}
			\cup\bigl\{\{b\}:b\in N\bigr\}.
		\]
		Equivalently, $\sigma[A]\in P(N,\mathcal G)$ for every 
		$A\in\mathcal F$.
	\end{enumerate}
\end{lemma}
\begin{proof}
	Write $B_m$ for the dual poset of $\BB_m$, that is, $B_m$
	is the poset consisting of one bottom element and $m$ 
	distinct maximal elements.

	(1)$\Rightarrow$(2) Assume  (1), that is,  
	$\varepsilon(P(N,\mathcal G))\in \Q\bigl(\BB_m,
	\varepsilon(P(M,\mathcal F))\bigr)$.
	By Lemma \ref{1.7} there is a surjective pp-morphism
	\[
		h: \biguplus_{i}P_i(M,\mathcal F) \uplus \biguplus_{j}B_m^j
		\ \twoheadrightarrow \ P(N, \mathcal G),
	\]
	where every $P_i(M,\mathcal F)$ is a copy of 
	$P(M,\mathcal F)$ and every $B_m^j$ is a copy 
	of $B_m$. Since $h$ is surjective, some element of the 
	disjoint union must map to the bottom element $N$
	of $P(N,\mathcal{G})$. Let this element be $x$. 
	The bottom point $N$ of $P(N,\mathcal G)$ has exactly $m+1$ maximal 
	points above it. Since $h$ is a pp-morphism,
	$h[\M(x)]=\M(h(x))=\M(N)$,	and therefore $x$ has at least $m+1$ 
	maximal points above it. No point of a copy of $B_m$ has more than 
	$m$ maximal points above it, moreover, in a copy of $P(M,\mathcal F)$ 
	the only point having $m+1$ maximal points above it is the 
	bottom point $M$. Consequently, $x$ is the bottom point of some copy
	of $P(M,\mathcal F)$. Restricting $h$ to this copy gives a pp-morphism
	\[ f:P(M,\mathcal F)\to P(N,\mathcal G)
	\qquad\text{such that}\qquad f(M)=N. \]
	By the pp-property, maximal elements are mapped onto maximal elements, 
	hence there is a map $\sigma:M\to N$ such that
	\[
		f(\{a\})=\{\sigma(a)\} \qquad(a\in M).
	\]
	Applying the pp-property again, we obtain
	\[ \bigl\{\{b\}:b\in N\bigr\} =\M(N)= \M(f(M)) =f[\M(M)]
	=\bigl\{\{\sigma(a)\}:a\in M\bigr\}.\]
	Thus $\sigma:M\to N$ is surjective. Since $|M|=|N|=m+1$ is finite,
	$\sigma$ is a bijection. For any $A\in \mathcal{F}$, the 
	pp-property gives
	\[
		\M(f(A))=f[\M(A)]=f\bigl[\{\{a\}:a\in A\}\bigr]
		=\bigl\{\{\sigma(a)\}:a\in A\bigr\}.
	\]
	Points of reduced posets are uniquely determined by the
	set of maximal points above them. Hence, 
	$f(A)=\sigma[A]$. In particular, $\sigma[A]\in P(N,\mathcal G)$.\\
	
	(2)$\Rightarrow$(1) Suppose that there is a bijection 
	$\sigma:M\to N$ such that
	\[
		\sigma[A]\in P(N,\mathcal G)\qquad\text{for every }A\in\mathcal F.
	\]
	Define
	\[
		f:P(M,\mathcal F)\to P(N,\mathcal G),
		\qquad f(X)=\sigma[X].
	\]
	This $f$ is well defined: Indeed,
	\[ f(M)=N, \qquad f(\{a\})=\{\sigma(a)\}, \]
	and the assumption (2) guarantees that $f(A)\in P(N,\mathcal G)$ 
	for every $A\in\mathcal F$.

	The map $f$ is order-preserving, because direct images 
	preserve inclusion. Moreover, for every $X\in P(M,\mathcal F)$ we have
	\[
		f[\M(X)]=f\bigl[\{\{a\}:a\in X\}\bigr]=
		\bigl\{\{\sigma(a)\}:a\in X\bigr\} = \M(\sigma[X])
   	 	=\M(f(X)).
	\]
	Hence, $f$ is a pp-morphism. The image of $f$ 
	contains the bottom point $N$ and all maximal points of $P(N,\mathcal G)$.
	
	Take any $Y\in P(N,\mathcal G)\setminus f[P(M,\mathcal F)]$ (if exists).
	Then $Y$ is neither $N$ nor a singleton. Consequently, 
	$2\leq |Y|\leq m$. Let $B$ be the $m$-element base of $B_m$, and 
	choose a surjection $\tau_Y:B\twoheadrightarrow Y$. Define
	\[
		g_Y:B_m\to P(N,\mathcal G), \qquad g_Y(B) = Y, \quad
		g_Y(\{a\})=\{\tau_Y(a)\} \quad (a\in B)
	\]
	The map $g_Y$ is order-preserving. Moreover, $g_Y$ is a pp-morphism:
	it is automatic at the maximal points of $B_m$, and calculations
	show
	\begin{align*}
		g_Y[\M(B)] = \bigl\{\{\tau_Y(a)\}:a\in B\bigr\}
   	 			  = \bigl\{\{y\}:y\in Y\bigr\}
				  = \M(Y)
				  =\M(g_Y(B)).
	\end{align*}
	Thus $g_Y$ is a pp-morphism whose image contains $Y$.

	Taking the disjoint union of $f$ and all the maps $g_Y$, for
	$Y\in P(N,\mathcal G)\setminus f[P(M,\mathcal F)]$, 
	gives a surjective pp-morphism
	\[
		P(M,\mathcal F)\uplus
		\biguplus_{Y}B_m \twoheadrightarrow P(N,\mathcal G).
	\]
	By Lemma \ref{1.7}, $\varepsilon(P(N,\mathcal G))\in
	\Q\bigl(\BB_m,\varepsilon(P(M,\mathcal F))\bigr)$.
\end{proof}

\section{Subvariety covers of $\Pa_m$}

\def\D{\mathbf{D}}
\begin{lemma}\label{vanveges}
	Let $\K$ be a quasivariety of p-algebras such that
	$\Pa_m\subsetneq \K$. Then there is a finite p-algebra $\D$
	such that $\Pa_m\subsetneq\Q(\D)\subseteq\K$.
\end{lemma}
\begin{proof}
	Let $\A, \B\in\Pa$ be arbitrary non-trivial p-algebras.
	
	\noindent Claim 1: There are homomorphisms $h:\A\to\B$ and $k:\B\to \A$. To see this
	consider the two-element $\BB_0\in\Pa$ (which is the two-element Boolean algebra). 
	By Glivenko's theorem (see \cite{beazer1992some}) there is a homomorphism
	from $\A$ into a non-trivial Boolean algebra, and from any non-trivial 
	Boolean algebra there is a 
	homomorphism into the two-element Boolean algebra. Composition of homomorphisms 
	gives a 
	homomorphism $\A\to\BB_0$. Also, $\BB_0$ is a subalgebra of $\B$, therefore 
	the composition 
	$\A\to\BB_0\hookrightarrow\B$ gives $h:\A\to\B$. The other case is analogous. \\
	
	\noindent Claim 2: $\Q(\A,\B) = \Q(\A\times\B)$. That $\A\times\B\in \Q(\A,\B)$ 
	and thus
	$\Q(\A\times\B)\subseteq \Q(\A,\B)$ is straightforward. For the converse inclusion
	define
	$f:\A\to\A\times \B$ by $f(a) = (a, h(a))$. Then $f$ is a homomorphism and as it is
	the identity 
	on the first coordinate, $f$ is injective. Therefore, 
	$\A\in \mathbf{IS}(\A\times \B)$. 
	Similarly, $\B\in\mathbf{IS}(\A\times\B)$. Consequently, 
	$\Q(\A, \B)\subseteq \Q(\A\times\B)$.\\
	
	To complete the proof, let $\K=\Q(\{\A_i: i\in I\})\supsetneq\Pa_m$.
	Then there is 
	$\A_i\notin \Pa_m$. As $\Pa_m$ is axiomatized by a single 
	identity $\mathbf{ib}_m$, 
	it follows that $\A_i\notmodels \mathbf{ib}_m$, and there are finitely 
	many elements
	$a_1, \ldots, a_n\in\A_i$ such that in the generated subalgebra 
	$\A=\Sg^{\A_i}(\{a_1, \ldots, a_n\})$ we also have $\A\notmodels\mathbf{ib}_m$,
	that is, $\A\notin\Pa_m$. The class $\Pa$ is locally finite, ensuring $\A$ 
	is finite. As $\Pa_m$ is generated by $\BB_m$, it follows that
	\[
		\Pa_m = \Q(\BB_m)\subsetneq \Q(\BB_m, \A)\subseteq \K\,.
	\]
	By claim 2 above, $\Q(\BB_m, \A) = \Q(\BB_m\times \A)$. 
	But $\BB_m\times \A$ is finite, thus putting 
	$\D=\BB_m\times\A$ concludes the proof.
\end{proof}


\begin{definition}
	For finite $m$ let $U_m$ be the poset
	\[
		U_m = \bigl(\wp(m)\setminus\{\emptyset\},\supseteq\bigr)
		\ =\ 
		P\bigl(m,\wp(m)\setminus\{\emptyset\}\bigr)\,.
	\]	
	Denote the dual algebra of $U_m$ by $\U_m$, that is, 
	$\U_m = \eps(U_m)$.	The posets $U_m$ for $m=1, \ldots, 4$ are 
	illustrated in Figure \ref{fig:um}.
\end{definition}

\begin{figure}[ht!]
    \begin{center}
		\begin{tikzpicture}[
		    x=1cm,
		    y=1cm,
		    dot/.style={circle,fill,inner sep=1.6pt}
		]

		\begin{scope}[xshift=0cm,yshift=0cm]
		    \node[dot] (u1-1) at (0,0) {};
			\draw (0,-0.7) node {$U_1$};
		\end{scope}

		\begin{scope}[xshift=3.5cm,yshift=0cm]
		    \node[dot] (u2-1)  at (-0.6,0.8) {};
		    \node[dot] (u2-2)  at ( 0.6,0.8) {};
		    \node[dot] (u2-12) at (0,0) {};

		    \draw (u2-12) -- (u2-1);
		    \draw (u2-12) -- (u2-2);
			
			\draw (0,-0.7) node {$U_2$};
		\end{scope}

		\begin{scope}[xshift=8cm,yshift=0cm]
		    \node[dot] (u3-1) at (-1.2,1.6) {};
		    \node[dot] (u3-2) at ( 0,1.6) {};
		    \node[dot] (u3-3) at ( 1.2,1.6) {};

		    \node[dot] (u3-12) at (-1.2,0.7) {};
		    \node[dot] (u3-13) at ( 0,0.7) {};
		    \node[dot] (u3-23) at ( 1.2,0.7) {};

		    \node[dot] (u3-123) at (0,-0.2) {};

		    \draw (u3-12) -- (u3-1);
		    \draw (u3-12) -- (u3-2);
		    \draw (u3-13) -- (u3-1);
		    \draw (u3-13) -- (u3-3);
		    \draw (u3-23) -- (u3-2);
		    \draw (u3-23) -- (u3-3);

		    \draw (u3-123) -- (u3-12);
		    \draw (u3-123) -- (u3-13);
		    \draw (u3-123) -- (u3-23);
			
			\draw (0,-0.7) node {$U_3$};
			
		\end{scope}

		\begin{scope}[xshift=4.8cm,yshift=-5.2cm]
		    \node[dot] (u4-1) at (-3,3) {};
		    \node[dot] (u4-2) at (-1,3) {};
		    \node[dot] (u4-3) at ( 1,3) {};
		    \node[dot] (u4-4) at ( 3,3) {};

		    \node[dot] (u4-12) at (-3.0,1.9) {};
		    \node[dot] (u4-13) at (-1.8,1.9) {};
		    \node[dot] (u4-14) at (-0.6,1.9) {};
		    \node[dot] (u4-23) at ( 0.6,1.9) {};
		    \node[dot] (u4-24) at ( 1.8,1.9) {};
		    \node[dot] (u4-34) at ( 3.0,1.9) {};

		    \node[dot] (u4-123) at (-2.25,0.8) {};
		    \node[dot] (u4-124) at (-0.75,0.8) {};
		    \node[dot] (u4-134) at ( 0.75,0.8) {};
		    \node[dot] (u4-234) at ( 2.25,0.8) {};

		    \node[dot] (u4-1234) at (0,-0.3) {};

		    \draw (u4-12) -- (u4-1);
		    \draw (u4-12) -- (u4-2);

		    \draw (u4-13) -- (u4-1);
		    \draw (u4-13) -- (u4-3);

		    \draw (u4-14) -- (u4-1);
		    \draw (u4-14) -- (u4-4);

		    \draw (u4-23) -- (u4-2);
		    \draw (u4-23) -- (u4-3);

		    \draw (u4-24) -- (u4-2);
		    \draw (u4-24) -- (u4-4);

		    \draw (u4-34) -- (u4-3);
		    \draw (u4-34) -- (u4-4);

		    \draw (u4-123) -- (u4-12);
		    \draw (u4-123) -- (u4-13);
		    \draw (u4-123) -- (u4-23);

		    \draw (u4-124) -- (u4-12);
		    \draw (u4-124) -- (u4-14);
		    \draw (u4-124) -- (u4-24);

		    \draw (u4-134) -- (u4-13);
		    \draw (u4-134) -- (u4-14);
		    \draw (u4-134) -- (u4-34);

		    \draw (u4-234) -- (u4-23);
		    \draw (u4-234) -- (u4-24);
		    \draw (u4-234) -- (u4-34);

		    \draw (u4-1234) -- (u4-123);
		    \draw (u4-1234) -- (u4-124);
		    \draw (u4-1234) -- (u4-134);
		    \draw (u4-1234) -- (u4-234);
			
			\draw (0,-0.7) node {$U_4$};
			
		\end{scope}

		\end{tikzpicture}

		\caption{The posets $U_m$ for $m=1, \ldots, 4$.}
		\label{fig:um}
	\end{center}
\end{figure}

\begin{theorem}\label{thm:unique-cover}
	For every $m\geq 1$, the variety $\Pa_m$ has exactly one cover in the
	lattice of quasivarieties of p-algebras, and this cover is
	\[ \Pa_m \prec \Q\bigl(\BB_m, \U_{m+1} \bigr)\,. \]
\end{theorem}
\begin{proof}
	The inclusion $\Pa_m \subseteq \Q\bigl(\BB_m, \U_{m+1} \bigr)$ is clear
	because $\Pa_m=\Q(\BB_m)$. The bottom point of $U_{m+1}$ has 
	$m+1$ maximal points above it, hence Lemma \ref{lem:mplus1} gives
	$\U_{m+1}\notin \Pa_m$, and thus the inclusion
	$\Pa_m \subsetneq \Q\bigl(\BB_m, \U_{m+1} \bigr)$ is proper.
	We show that $\Q\bigl(\BB_m, \U_{m+1} \bigr)$ is contained in 
	every quasivariety properly containing $\Pa_m$. 
	
	Let $\K$ be a quasivariety such that
	$\Pa_m\subsetneq\K$. By Lemma \ref{vanveges}, 
	there is a finite $\D$ such that
	\[
		\Pa_m\subsetneq\Q(\D)\subseteq\K.
	\]
	Apply Lemma \ref{lemma-1} to $\D$: There are a set $N$ 
	of cardinality $m+1$ and a family $\mathcal G$ of nonempty subsets of $N$ 
	such that, for $\eps(P(N,\calG))\subseteq\D$. As the element $N$ of $P(N,\calG)$
	has $m+1$ maximal elements above it, by Lemma \ref{lem:mplus1}, 
	$\eps(P(N,\calG))\notin \Pa_m$. Therefore
	\[
		\Pa_m\subsetneq \Q(\BB_m, \eps(P(N,\calG)))\subseteq \Q(\D)\subseteq \K\,.
	\]
	
	Write $U_{m+1} = P(M, \wp(M)\setminus\{\emptyset\})$ for some set $|M|=m+1$. 
	Choose any bijection $\sigma:N\to M$.
	For every $A\in\mathcal G$, the set $A$ is nonempty, and hence
	$\sigma[A]$ is a nonempty subset of $M$. Therefore
	\[
		\sigma[A]\in U_{m+1} \qquad\text{for every }A\in\mathcal G.
	\]
	Lemma \ref{lem:relative-comparison} gives
	$\Q(\BB_m, \U_{m+1})\ \subseteq \ \Q(\BB_m, \eps(P(N,\calG)))$,
	and therefore
	\[
		\Pa_m\subsetneq \Q(\BB_m, \U_{m+1}) \subseteq
		\Q(\BB_m, \eps(P(N,\calG)))\subseteq \Q(\D)\subseteq \K
	\]
	as desired.
\end{proof}

\begin{remark}
	Note that $\delta(\BB_0) = U_1$ and $\delta(\BB_2)=U_2$. It follows that
	\[
	\Q(\BB_0, \U_1) = \Q(\BB_0) = \Pa_0,\quad
	\text{ and }\quad
	\Q(\BB_1, \U_2) = \Q(\BB_1,\BB_2) = \Pa_2
	\]
	It is known that the interval $[\Pa_{-1}, \Pa_2]$ is a $4$-element chain, 
	see \cite{KowSlom}. Thus, $\Pa_2$ is the cover of $\Pa_1$ in Theorem 
	\ref{thm:unique-cover}.
\end{remark}

\begin{figure}[ht!]
    \begin{center}
		\begin{tikzpicture}[
		    x=1cm,
		    y=1cm,
		    dot/.style={circle,fill,inner sep=1.6pt}
		]

		\begin{scope}[xshift=3.5cm,yshift=0cm]
			\node[dot] (minus1) at (0,0) {}; 
			\node[dot] (nulla) at (0,1) {}; %
			\node[dot] (egy) at (0,2) {}; %
			\node[dot] (ketto) at (0,3) {};%
			\node[dot] (Q1) at (0,4) {}; %
			\node[dot] (harom) at (0,5) {}; %
			\node[dot] (Q2) at (0,6) {}; %

		    \draw (minus1) -- (nulla) -- (egy) -- (ketto) -- (Q1);
			\draw[dashed] (Q1) to[out=30, in=-30] (harom);
			\draw[dashed] (Q1) to[out=150, in=-150] (harom);
			\draw (harom) -- (Q2);
			
			\draw (minus1) node[right] {$\Pa_{-1} = \mathsf{triv}$};
			\draw (nulla) node[right] {$\Pa_0 = \mathsf{BA} = \Q(\BB_0, \U_1)$};
			\draw (egy) node[right] {$\Pa_1 = \Q(\BB_1)$};
			\draw (ketto) node[right] {$\Pa_2 = \Q(\BB_1, \U_2) = \Q(\BB_2)$}; 
			\draw (Q1) node[right] {$\Q(\BB_2, \U_3)$};
			\draw (harom) node[right] {$\Pa_3 = \Q(\BB_3)$};
			\draw (Q2) node[right] {$\Q(\BB_3, \U_4)$};
			
			\draw (Q2) node[above] {$\vdots$};
		\end{scope}
		
		\end{tikzpicture}
		
		\caption{The bottom of the quasivariety lattice of p-algebras. Solid lines represent covers.}
		\label{fig:bottom}
	\end{center}
\end{figure}

\section{The case $m=2$}\label{mketto}

The recent paper \cite{KowSlom} gives the $3$ finite posets depicted in Figure \ref{fig:3poset} below and claims that any of these three posets generate a cover of $\Pa_2$.
This claim appears incorrect and is in contradiction with Theorem \ref{thm:unique-cover}, according to which the unique cover of $\Pa_2$ is $\Q(\BB_2, \U_3)$. To resolve this conflict, let us recheck the case $m=2$.

    \begin{figure}[!ht]
    \begin{center}
        \begin{tikzpicture}[thick,scale=0.95]
			\node[label=above:{\scriptsize $\{1\}$}] (A1) at (0,0) {};
			\node[label=above:{\scriptsize $\{2\}$}] (B1) at (1,0) {};
			\node[label=above:{\scriptsize $\{3\}$}] (C1) at (2,0) {};

			\node[label=left:{\scriptsize $\{1,2\}$}] (D1) at (0.5,-0.7) {};
			\node[label=below:{\scriptsize $\{1,2,3\}$}] (E1) at (1,-1.4) {};

			\draw [fill] (A1) circle [radius=.06];
			\draw [fill] (B1) circle [radius=.06];
			\draw [fill] (C1) circle [radius=.06];
			\draw [fill] (D1) circle [radius=.06];
			\draw [fill] (E1) circle [radius=.06];
			
			\draw (E1.center) -- (D1.center) -- (A1.center);
			\draw (D1.center) -- (B1.center);
			\draw (E1.center) -- (C1.center);
			
			\draw (1,-2.5) node {$P$};
			
			
			\node[label=above:{\scriptsize $\{1\}$}] (A2) at (0+4,0) {};
			\node[label=above:{\scriptsize $\{2\}$}] (B2) at (1+4,0) {};
			\node[label=above:{\scriptsize $\{3\}$}] (C2) at (2+4,0) {};

			\node[label=left:{\scriptsize $\{1,2\}$}] (D2) at (0.5+4,-0.7) {};
			\node[label=right:{\scriptsize $\{2,3\}$}] (D22) at (1+0.5+4,-0.7) {};
			\node[label=below:{\scriptsize $\{1,2,3\}$}] (E2) at (1+4,-1.4) {};

			\draw [fill] (A2) circle [radius=.06];
			\draw [fill] (B2) circle [radius=.06];
			\draw [fill] (C2) circle [radius=.06];
			\draw [fill] (D2) circle [radius=.06];
			\draw [fill] (D22) circle [radius=.06];
			\draw [fill] (E2) circle [radius=.06];
			
			\draw (E2.center) -- (D2.center) -- (A2.center);
			\draw (D2.center) -- (B2.center);
			\draw (E2.center) -- (D22.center);
			\draw (D22.center) -- (B2.center);
			\draw (D22.center) -- (C2.center);
			
			\draw (1+4,-2.5) node {$Q$};
			
			
			\node[label=above:{\scriptsize $\{1\}$}] (A3) at (0+4+4,0) {};
			\node[label=above:{\scriptsize $\{2\}$}] (B3) at (1+4+4,0) {};
			\node[label=above:{\scriptsize $\{3\}$}] (C3) at (2+4+4,0) {};

			\node[label=left:{\scriptsize $\{1,2\}$}] (D3) at (0+4+4,-0.7) {};
			\node[label=right:{\tiny $\{1,3\}$}] (D32) at (1+4+4,-0.7) {};
			\node[label=right:{\scriptsize $\{2,3\}$}] (D33) at (2+4+4,-0.7) {};

			\node[label=below:{\scriptsize $\{1,2,3\}$}] (E3) at (1+4+4,-1.4) {};

			\draw [fill] (A3) circle [radius=.06];
			\draw [fill] (B3) circle [radius=.06];
			\draw [fill] (C3) circle [radius=.06];
			\draw [fill] (D3) circle [radius=.06];
			\draw [fill] (D32) circle [radius=.06];
			\draw [fill] (D33) circle [radius=.06];
			\draw [fill] (E3) circle [radius=.06];
			
			\draw (E3.center) -- (D3.center) -- (A3.center);
			\draw (D3.center) -- (B3.center);
			\draw (E3.center) -- (D32.center);
			\draw (D32.center) -- (A3.center);
			\draw (D32.center) -- (C3.center);
			\draw (E3.center) -- (D33.center);
			\draw (D33.center) -- (C3.center);
			\draw (D33.center) -- (B3.center);

			\draw (1+4+4,-2.5) node {$R$};

		\end{tikzpicture}
		\caption{The three reduced posets $P(M, \calF)$ with 
		$M=\{1,2,3\}$, and $\calF\neq\emptyset$, up to isomorphisms.}
		\label{fig:3poset}
	\end{center}
	\end{figure}

\noindent The posets from Figure \ref{fig:3poset} are reduced posets 
with $M = \{1,2,3\}$ and
\begin{align*}
	P &\ =\  P(M, \big\{\{1,2\}\big\}), \\
	Q &\ =\  P(M, \big\{\{1,2\}, M\setminus \{1\}\big\}), \text{ and }\\
	R &\ =\  P(M, \big\{\{1,2\}, M\setminus \{1\}, M\setminus \{2\}\big\}).
\end{align*} 
Clearly, $P(M, \emptyset) = \delta(\BB_3)$.

\begin{theorem}\label{thm:uccso-elotti}
	$\Pa_2\ \subsetneq\  \Q(\eps(R))\ \subsetneq \ 
	\Q(\eps(Q))\ \subsetneq\ \Q(\eps(P))$. In particular, the posets 
	$P$ and $Q$ do not generate a cover of $\Pa_2$.
\end{theorem}
\begin{proof}
	For every $X\in\{P, Q, R\}$ there is $x\in X$ with $|\M(x)|=3$, thus
	$\Q(\eps(X))\neq \Pa_2$ by Lemma \ref{lem:mplus1}. It is also
	easy to check that there is a surjective pp-morphism 
	$X\twoheadrightarrow\delta(\BB_2)$, thus $\Pa_2\subseteq\Q(\eps(X))$ 
	by Lemma \ref{lem:egykomponens}. 
	Figure \ref{fig:3poset-pp} shows surjective pp-morphisms 
	$h:P\uplus P\twoheadrightarrow Q$ and $k:Q\uplus Q\twoheadrightarrow R$.
	By composition, there is $i:P\uplus P\uplus P\uplus P\twoheadrightarrow R$.
	By Lemma \ref{1.7} it follows that
	\[
		\Q(\eps(R))\ \subseteq \ \Q(\eps(Q))\ \subseteq\ \Q(\eps(P))\,.
	\]
	To see that these inclusions are proper, assume there was 
	$h: \biguplus Q\twoheadrightarrow P$. 
	By surjectivity there is a component of the disjoint union and a point 
	$x\in Q$ such that $h(x)$ is the bottom element of $P$. By maxima preservation
	$x$ must also be the bottom element of $Q$, because that is the only point with 
	exactly $3$ maxima above it. Then the maxima in this copy of $Q$ are mapped
	bijectively onto the maxima of $P$. Finally, the points in $Q$ with two maxima
	should be mapped to the only point in $P$ with two maxima. This means that 
	the two points in $Q$ with two maxima are collapsed, and that 
	contradicts
	$h$ being a pp-morphism. See the left side of Figure \ref{fig:3poset-pp-no}. The
	case that there is no $h: \biguplus R\twoheadrightarrow Q$ is analogous, 
	see right-hand side of Figure \ref{fig:3poset-pp-no}.
	It follows, by Lemma \ref{1.7}, that 
	\[
		\Q(\eps(R))\ \subsetneq \ \Q(\eps(Q))\ \subsetneq\ \Q(\eps(P))\,,
	\]
	in particular, $\Q(\eps(Q))$ and $\Q(\eps(P))$ cannot be covers of $\Pa_2$.
\end{proof}

    \begin{figure}[!ht]
    \begin{center}
    \begin{tikzpicture}[thick,scale=0.7]
			\node[label=above:{\scriptsize $a$}] (A1) at (0,0) {};
			\node[label=above:{\scriptsize $b$}] (B1) at (1,0) {};
			\node[label=above:{\scriptsize $c$}] (C1) at (2,0) {};

			\node[label=left:{\scriptsize $d$}] (D1) at (0.5,-0.7) {};
			\node[label=below:{\scriptsize $f$}] (E1) at (1,-1.4) {};

			\draw [fill] (A1) circle [radius=.06];
			\draw [fill] (B1) circle [radius=.06];
			\draw [fill] (C1) circle [radius=.06];
			\draw [fill] (D1) circle [radius=.06];
			\draw [fill] (E1) circle [radius=.06];
			
			\draw (E1.center) -- (D1.center) -- (A1.center);
			\draw (D1.center) -- (B1.center);
			\draw (E1.center) -- (C1.center);

			
			\node[label=above:{\scriptsize $c$}] (A12) at (0+3,0) {};
			\node[label=above:{\scriptsize $b$}] (B12) at (1+3,0) {};
			\node[label=above:{\scriptsize $a$}] (C12) at (2+3,0) {};

			\node[label=left:{\scriptsize $e$}] (D12) at (0.5+3,-0.7) {};
			\node[label=below:{\scriptsize $f$}] (E12) at (1+3,-1.4) {};

			\draw [fill] (A12) circle [radius=.06];
			\draw [fill] (B12) circle [radius=.06];
			\draw [fill] (C12) circle [radius=.06];
			\draw [fill] (D12) circle [radius=.06];
			\draw [fill] (E12) circle [radius=.06];
			
			\draw (E12.center) -- (D12.center) -- (A12.center);
			\draw (D12.center) -- (B12.center);
			\draw (E12.center) -- (C12.center);

			\node[label=above:{\scriptsize $a$}] (A2) at (0+4+5,0) {};
			\node[label=above:{\scriptsize $b$}] (B2) at (1+4+5,0) {};
			\node[label=above:{\scriptsize $c$}] (C2) at (2+4+5,0) {};

			\node[label=left:{\scriptsize $d$}] (D2) at (0.5+4+5,-0.7) {};
			\node[label=right:{\scriptsize $e$}] (D22) at (1+0.5+4+5,-0.7) {};
			\node[label=below:{\scriptsize $f$}] (E2) at (1+4+5,-1.4) {};

			\draw [fill] (A2) circle [radius=.06];
			\draw [fill] (B2) circle [radius=.06];
			\draw [fill] (C2) circle [radius=.06];
			\draw [fill] (D2) circle [radius=.06];
			\draw [fill] (D22) circle [radius=.06];
			\draw [fill] (E2) circle [radius=.06];
			
			\draw (E2.center) -- (D2.center) -- (A2.center);
			\draw (D2.center) -- (B2.center);
			\draw (E2.center) -- (D22.center);
			\draw (D22.center) -- (B2.center);
			\draw (D22.center) -- (C2.center);
						
			
			\node[label=above:{\footnotesize $h$}] at (7, -0.7) {$\twoheadrightarrow$};
			\node at (2.5, -1.3) {$\uplus$};
			
		
			\node[label=above:{\scriptsize $b$}] (A2) at (0,0-3.5) {};
			\node[label=above:{\scriptsize $a$}] (B2) at (1,0-3.5) {};
			\node[label=above:{\scriptsize $c$}] (C2) at (2,0-3.5) {};

			\node[label=left:{\scriptsize $d$}] (D2) at (0.5,-0.7-3.5) {};
			\node[label=right:{\scriptsize $e$}] (D22) at (1+0.5,-0.7-3.5) {};
			\node[label=below:{\scriptsize $g$}] (E2) at (1,-1.4-3.5) {};

			\draw [fill] (A2) circle [radius=.06];
			\draw [fill] (B2) circle [radius=.06];
			\draw [fill] (C2) circle [radius=.06];
			\draw [fill] (D2) circle [radius=.06];
			\draw [fill] (D22) circle [radius=.06];
			\draw [fill] (E2) circle [radius=.06];
		
			\draw (E2.center) -- (D2.center) -- (A2.center);
			\draw (D2.center) -- (B2.center);
			\draw (E2.center) -- (D22.center);
			\draw (D22.center) -- (B2.center);
			\draw (D22.center) -- (C2.center);
		
		
			\node[label=above:{\scriptsize $a$}] (A2) at (0+3,0-3.5) {};
			\node[label=above:{\scriptsize $c$}] (B2) at (1+3,0-3.5) {};
			\node[label=above:{\scriptsize $b$}] (C2) at (2+3,0-3.5) {};

			\node[label=left:{\scriptsize $e$}] (D2) at (0.5+3,-0.7-3.5) {};
			\node[label=right:{\scriptsize $f$}] (D22) at (1+0.5+3,-0.7-3.5) {};
			\node[label=below:{\scriptsize $g$}] (E2) at (1+3,-1.4-3.5) {};

			\draw [fill] (A2) circle [radius=.06];
			\draw [fill] (B2) circle [radius=.06];
			\draw [fill] (C2) circle [radius=.06];
			\draw [fill] (D2) circle [radius=.06];
			\draw [fill] (D22) circle [radius=.06];
			\draw [fill] (E2) circle [radius=.06];
		
			\draw (E2.center) -- (D2.center) -- (A2.center);
			\draw (D2.center) -- (B2.center);
			\draw (E2.center) -- (D22.center);
			\draw (D22.center) -- (B2.center);
			\draw (D22.center) -- (C2.center);

		
			\node[label=above:{\scriptsize $a$}] (A3) at (0+3+3+3,0-3.5) {};
			\node[label=above:{\scriptsize $b$}] (B3) at (1+3+3+3,0-3.5) {};
			\node[label=above:{\scriptsize $c$}] (C3) at (2+3+3+3,0-3.5) {};

			\node[label=left:{\scriptsize $d$}] (D3) at (0+3+3+3,-0.7-3.5) {};
			\node[label=right:{\tiny $e$}] (D32) at (1+3+3+3,-0.7-3.5) {};
			\node[label=right:{\scriptsize $f$}] (D33) at (2+3+3+3,-0.7-3.5) {};

			\node[label=below:{\scriptsize $g$}] (E3) at (1+3+3+3,-1.4-3.5) {};

			\draw [fill] (A3) circle [radius=.06];
			\draw [fill] (B3) circle [radius=.06];
			\draw [fill] (C3) circle [radius=.06];
			\draw [fill] (D3) circle [radius=.06];
			\draw [fill] (D32) circle [radius=.06];
			\draw [fill] (D33) circle [radius=.06];
			\draw [fill] (E3) circle [radius=.06];
		
			\draw (E3.center) -- (D3.center) -- (A3.center);
			\draw (D3.center) -- (B3.center);
			\draw (E3.center) -- (D32.center);
			\draw (D32.center) -- (A3.center);
			\draw (D32.center) -- (C3.center);
			\draw (E3.center) -- (D33.center);
			\draw (D33.center) -- (C3.center);
			\draw (D33.center) -- (B3.center);

			\node[label=above:{\footnotesize $k$}] at (7, -0.7-3.5) {$\twoheadrightarrow$};
			\node at (2.5, -1.3-3.5) {$\uplus$};
	
	\end{tikzpicture}
	\caption{Surjective pp-morphisms $h:P\uplus P\twoheadrightarrow Q$ and 
	$k:Q\uplus Q\twoheadrightarrow R$. Points with the same labels are identified.}
	\label{fig:3poset-pp}
	\end{center}
	\end{figure}

    \begin{figure}[!ht]
    \begin{center}
    \begin{tikzpicture}[thick,scale=0.7]
			\node[label=above:{\scriptsize $a$}] (A1) at (0+5,0) {};
			\node[label=above:{\scriptsize $b$}] (B1) at (1+5,0) {};
			\node[label=above:{\scriptsize $c$}] (C1) at (2+5,0) {};

			\node[label=left:{\scriptsize $e=d$}] (D1) at (0.5+5,-0.7) {};
			\node[label=below:{\scriptsize $f$}] (E1) at (1+5,-1.4) {};

			\draw [fill] (A1) circle [radius=.06];
			\draw [fill] (B1) circle [radius=.06];
			\draw [fill] (C1) circle [radius=.06];
			\draw [fill] (D1) circle [radius=.06];
			\draw [fill] (E1) circle [radius=.06];
			
			\draw (E1.center) -- (D1.center) -- (A1.center);
			\draw (D1.center) -- (B1.center);
			\draw (E1.center) -- (C1.center);

			\node[label=above:{\scriptsize $a$}] (A2) at (0,0) {};
			\node[label=above:{\scriptsize $b$}] (B2) at (1,0) {};
			\node[label=above:{\scriptsize $c$}] (C2) at (2,0) {};

			\node[label=left:{\scriptsize $d$}] (D2) at (0.5,-0.7) {};
			\node[label=right:{\scriptsize $e$}] (D22) at (1+0.5,-0.7) {};
			\node[label=below:{\scriptsize $f$}] (E2) at (1,-1.4) {};

			\draw [fill] (A2) circle [radius=.06];
			\draw [fill] (B2) circle [radius=.06];
			\draw [fill] (C2) circle [radius=.06];
			\draw [fill] (D2) circle [radius=.06];
			\draw [fill] (D22) circle [radius=.06];
			\draw [fill] (E2) circle [radius=.06];
			
			\draw (E2.center) -- (D2.center) -- (A2.center);
			\draw (D2.center) -- (B2.center);
			\draw (E2.center) -- (D22.center);
			\draw (D22.center) -- (B2.center);
			\draw (D22.center) -- (C2.center);
						
			
			\node[label=above:{\footnotesize $h$}] at (3, -0.7) {$\twoheadrightarrow$};

		
			\node[label=above:{\scriptsize $b$}] (A2) at (0+14,0) {};
			\node[label=above:{\scriptsize $a$}] (B2) at (1+14,0) {};
			\node[label=above:{\scriptsize $c$}] (C2) at (2+14,0) {};

			\node[label=left:{\scriptsize $f=d$}] (D2) at (0.5+14,-0.7) {};
			\node[label=right:{\scriptsize $e$}] (D22) at (1+0.5+14,-0.7) {};
			\node[label=below:{\scriptsize $g$}] (E2) at (1+14,-1.4) {};

			\draw [fill] (A2) circle [radius=.06];
			\draw [fill] (B2) circle [radius=.06];
			\draw [fill] (C2) circle [radius=.06];
			\draw [fill] (D2) circle [radius=.06];
			\draw [fill] (D22) circle [radius=.06];
			\draw [fill] (E2) circle [radius=.06];
		
			\draw (E2.center) -- (D2.center) -- (A2.center);
			\draw (D2.center) -- (B2.center);
			\draw (E2.center) -- (D22.center);
			\draw (D22.center) -- (B2.center);
			\draw (D22.center) -- (C2.center);

		
			\node[label=above:{\scriptsize $a$}] (A3) at (0+3+3+3,0) {};
			\node[label=above:{\scriptsize $b$}] (B3) at (1+3+3+3,0) {};
			\node[label=above:{\scriptsize $c$}] (C3) at (2+3+3+3,0) {};

			\node[label=left:{\scriptsize $d$}] (D3) at (0+3+3+3,-0.7) {};
			\node[label=right:{\tiny $e$}] (D32) at (1+3+3+3,-0.7) {};
			\node[label=right:{\scriptsize $f$}] (D33) at (2+3+3+3,-0.7) {};

			\node[label=below:{\scriptsize $g$}] (E3) at (1+3+3+3,-1.4) {};

			\draw [fill] (A3) circle [radius=.06];
			\draw [fill] (B3) circle [radius=.06];
			\draw [fill] (C3) circle [radius=.06];
			\draw [fill] (D3) circle [radius=.06];
			\draw [fill] (D32) circle [radius=.06];
			\draw [fill] (D33) circle [radius=.06];
			\draw [fill] (E3) circle [radius=.06];
		
			\draw (E3.center) -- (D3.center) -- (A3.center);
			\draw (D3.center) -- (B3.center);
			\draw (E3.center) -- (D32.center);
			\draw (D32.center) -- (A3.center);
			\draw (D32.center) -- (C3.center);
			\draw (E3.center) -- (D33.center);
			\draw (D33.center) -- (C3.center);
			\draw (D33.center) -- (B3.center);

			\node[label=above:{\footnotesize $k$}] at (12.5, -0.7) {$\twoheadrightarrow$};

	\end{tikzpicture}
	\caption{No pp-morphisms $h:Q\twoheadrightarrow P$ and 
	$k:R\twoheadrightarrow Q$ mapping bottoms onto bottoms: $h(\M(e))\neq \M(h(e))$,
	$k(\M(f))\neq \M(k(f))$.}
	\label{fig:3poset-pp-no}
	\end{center}
	\end{figure}

\begin{theorem}\label{thm:uccso}
	$\Pa_2\prec \Q(\eps(R))$ is a cover.
\end{theorem}
\begin{proof}
	Theorem \ref{thm:unique-cover} says that the unique cover of 
	$\Pa_2$ is $\Q(\BB_2, \U_3)$. To prove the present statement
	we only need to check that $\Q(\BB_2, \U_3) = \Q(\eps(R))$. 
	Note that $R$ and $U_3$ are the same posets and thus we need
	$\Q(\BB_2, \U_3) = \Q(\U_3)$. Checking that any surjection $3\twoheadrightarrow 2$
	induces a surjective pp-morphism $U_3\twoheadrightarrow B_2$ is a routine and
	left to the reader. Applying Lemma \ref{1.7} concludes the proof. 
\end{proof}

The case $m=2$ is an important exception: for $m\geq 3$ the proof of
Theorem \ref{thm:uccso} can not be reproduced. In more details:

\begin{theorem}
	The inclusion $\Q(\U_{m+1})\subseteq \Q(\BB_m, \U_{m+1})$ is proper
	for every $m\geq 3$. For $m\leq 2$ we have 
	$\Q(\U_{m+1})= \Q(\BB_m, \U_{m+1})$.
\end{theorem}
\begin{proof}
	As for the case $m=0$: $U_1 = B_0$, hence $\Q(\U_1)=\Q(\BB_0, \U_1)$. 
	As for $m=1$: $U_2=B_2$, thus $\Q(\U_2) = \Q(\BB_2) = \Q(\BB_2, \BB_1)$. 
	The case $m=2$ is essentially Theorem \ref{thm:uccso}. 
	For $m\geq 3$ it is enough to show $\BB_m\notin \Q(\U_{m+1})$. 
	By way of contradiction suppose $\BB_m\in \Q(\U_{m+1})$. By Lemma \ref{1.7}
	there is a pp-morphism $f:U_{m+1}\twoheadrightarrow B_m$ whose image
	contains the bottom of $B_m$. Such a map must be surjective by the
	pp-condition. Restrict $f$ to the $m+1$ maximal points of $U_{m+1}$. This
	gives a surjection $g: m+1\twoheadrightarrow m$, and there must exist
	$a, b$ with $g(a)\neq g(b)$. The point $\{a, b\}\in U_{m+1}$ has exactly
	two maximal points above it, hence the pp-condition gives
	\[
		\M(f(\{a,b\})) = f[\M(\{a,b\})] = \{ g(a), g(b) \}.
	\]
	Therefore, $B_m$ should contaion a point having exactly two maximal points above it. 
	This is impossible if $m\geq 3$. 
	
\end{proof}

\subsection*{Funding}
Not applicable.

\subsection*{Ethical approval}
Not applicable. 

\subsection*{Competing interests} 
Not applicable. 

\subsection*{Authors' contributions} 
Not applicable.

\subsection*{Availability of data and materials}
Not applicable. 

\end{document}